\documentclass[reqno,12pt]{amsart}
\usepackage{amsmath,amsthm,amssymb}
\usepackage{tikz}
\usetikzlibrary{arrows}
\usetikzlibrary{backgrounds, calc, positioning}

\date{}
\newcommand{\RR}{\mathbb{R}}
\newcommand{\CC}{\mathbb{C}}
\newcommand{\ZZ}{\mathbb{Z}}

\newtheorem{theorem}{Theorem}[section]
\newtheorem{lemma}[theorem]{Lemma}
\newtheorem{corollary}[theorem]{Corollary}
\newtheorem{proposition}[theorem]{Proposition}

\newtheorem{remark}[theorem]{Remark}

\newtheorem{conjecture}[theorem]{Conjecture}

\numberwithin{equation}{section}

\def \vphi {\varphi}

\begin{document}

 \title{Extremal  K{\"a}hler-Einstein metric for two-dimensional convex bodies}

\author{Bo'az~Klartag}
\address{Department of Mathematics, Weizmann Institute of Science, Rehovot 76100, Israel and
	School of Mathematical Sciences, Tel Aviv University, Tel Aviv 69978, Israel
 }
\email{boaz.klartag@weizmann.ac.il}

\vspace{5mm}
\author{Alexander~V.~Kolesnikov}
\address{ Higher School of Economics, Moscow,  Russia}
\email{Sascha77@mail.ru}

\thanks{
The first named author was partially supported by a European Research Council (ERC) grant.
 The
second named author was supported by the Supported by RFBR
 project 17-01-00662 and  DFG project RO 1195/12-1. The article was prepared
within the framework of the Academic Fund Program at the National Research University Higher
School of Economics (HSE) in 2017--2018 (grant No 17-01-0102) and by the Russian Academic
Excellence Project ”5-100”.
This material is based upon work supported by the National Science Foundation under Grant No. DMS-1440140 while the authors were in residence at the Mathematical Sciences Research Institute in Berkeley, California, during the Fall 2017 semester.
}


 \begin{abstract}
Given a convex body $K \subset \mathbb{R}^n$  with the barycenter at the origin we consider  the corresponding
K{\"a}hler-Einstein equation $e^{-\Phi} = \det D^2 \Phi$. If $K$ is a simplex, then the Ricci tensor of the Hessian metric
 $D^2 \Phi$ is constant and equals $\frac{n-1}{4(n+1)}$. We conjecture  that the Ricci tensor of $D^2 \Phi$
 for arbitrary $K$ is uniformly bounded by  $\frac{n-1}{4(n+1)}$ and verify this conjecture in the two-dimensional case.
 The general case remains open.
\end{abstract}

\maketitle

\section{Introduction}

We consider a convex  body  $K \subset \mathbb{R}^n$ with the barycenter at the origin
and the associated equation of the Monge-Amp{\'e}re type
\begin{equation}
\label{KE}
e^{-\Phi} = \det D^2 \Phi,
\end{equation}
\begin{equation}
 \nabla \Phi(\mathbb{R}^n)=K.
\label{eq_1026} \end{equation}
The function  $\Phi \colon \mathbb{R}^n \to \mathbb{R}$ is assumed to be  convex.
Given a solution $\Phi$ to (\ref{KE}) one considers the naturally associated Hessian (Riemannian) metric $D^2 \Phi$
on $\mathbb{R}^n$, the so-called  K\"ahler-Einstein metric. After a suitable complexification this metric indeed
becomes a K\"ahler metric which is Einstein.

The interest in this equation  is motivated by
various problems of differential and algebraic  geometry. In patricular, equation (\ref{KE}) naturally arises in the theory of
toric varietes.
The situation when $K$ is a  rational convex  polytope is of particular interest,
because in this case $K$ is a moment polytope of a toric orbifold $M$.
 We refer to \cite{Aubin}, \cite{BeBe}, \cite{WangZhu}, \cite{Donaldson}, \cite{Legendre}
for more information on K\"ahler geometry, toric varietes and the importance of the equation
(\ref{KE}) in convex geometry.
The existence and uniqueness of a solution to (\ref{KE}) under various assumptions has been proved in a series of papers
\cite{WangZhu}, \cite{Donaldson},  \cite{Legendre}, \cite{BeBe}, \cite{CEKl}.

Another motivation for study (\ref{KE}) comes from convex analysis. There are deep reasons to believe that
equation (\ref{KE}) can contribute to understanding a number of
difficult open conjectures (KLS conjecture,  slicing problem) from asymptotic convex geometry (see \cite{K_moment},
\cite{K_part_I}, \cite{KolMil}).
See also \cite{Kol}, \cite{KK}, \cite{KK2}, where Hessian metrics have been considered in a more general  measure-transportational setting.  A comprehensive overview of classical and modern results as well as open problems  in convex analysis
the reader can find in \cite{AGM}.

An important example of  $K$ is given by the simplex
$$
S = \left \{ (x_1,\ldots,x_n) \in \RR^n \, ; \, \sum_{i=1}^{n} x_i \le 1, \  \forall i, x_i \ge -1 \right \}.
$$
It is known that for $K = S$  equation  (\ref{KE}) admits an explicit solution (see Section 4.3 below).
The corresponding Hessian metric $D^2 \Phi$ is  isometric to a spherical orthant
$\{ x \in \RR^n_+ \, ; \, \sum_{i=1}^n x_i^2 = 4 (n+1) \}$.
In particular, the corresponding Ricci tensor has the form
$$
{\rm{Ric}}_{S} = \frac{n-1}{4(n+1)} \cdot D^2 \Phi.
$$

Motivated by problems from convex geometry we suggest the following conjecture.

 \begin{conjecture}
 The Ricci curvature of $D^2 \Phi$ is bounded by $\frac{n-1}{4(n+1)}$. In particular, the largest  value is realized on $S$ uniformly.
 \end{conjecture}

The aim of this article is to provide an affirmative answer to this conjecture in the two-dimensional case.
 We pursue the approach initiated in the classical papers of Calabi  and prove the main resut by differentiating equation
(\ref{KE}) and applying the maximum principle.  Here we follow the computational technique developed in earlier papers \cite{Kol}, \cite{KK}, \cite{KK2}.
Finally, we present the computations for the simplex to demonstrate that our result is sharp, and also for the ball and the cube. The general problem remains open.

\section{ Notations and preliminary results}

We assume throughout that we are given the standard Euclidean coordinate system  $\{x_i\}$.
For an arbitrary convex body $K$ satisfying
$$
\int_{K} x_i dx =0
$$
we consider $\Phi$, the smooth solution  to (\ref{KE}) which is uniquely-determined up to translation.

The space $\mathbb{R}^n$ is equipped with the metric
$$
h = h_{ij} dx^i dx^j = \Phi_{ij} dx^i dx^j = (\partial^2_{x_i x_j} \Phi) dx^i dx^j
$$
and with the measure $\mu = e^{-\Phi} dx$.

We give below a list of useful computational formulas, the reader can find the proof in \cite{Kol}.
 It is convenient to use the following notation:
 $$
 \Phi_i = \partial_{x_i} \Phi, \ \Phi_{ij} = \partial^2_{x_i x_j} \Phi, \  \Phi_{ijk} = \partial^3_{x_i x_j x_k} \Phi
 $$
 We follow the standard conventions of  Riemannian geometry (i.e., $\Phi^{ij}$ is inverse to $\Phi_{ij}$, Einstein summation, raising indices etc.).

 The measure $\mu$ has the following density with respect to the Riemannian volume
 $$
 \mu = e^{-P} d vol_{h}, \ \ P = \frac{1}{2} \Phi.
 $$

 The associated diffusion generator (weighted Laplacian) $L$ has the form
  $$
L f = \Phi^{ij} f_{ij}.
 $$

 Differentiating the K{\"a}hler-Einstein equation, one gets
 the folowing important identity, that for any fixed index $i$, setting $f = \Phi_i$,
 \begin{equation}
 \label{eigenf}
 L f = \Phi_{iab} \Phi^{ab} = - \Phi_i.
 \end{equation}

 The following non-negative symmetric tensor $g$  plays prominent role in our analysis
 $$
 g_{ij} = \Phi_{iab} \Phi_j^{ab}.
 $$

Finally, we give a list  of formulas for the most important quantities (see  \cite{Kol}, \cite{KK2}).

 \begin{enumerate}
\item Connection
 $$
 \Gamma^k_{ij} = \frac{1}{2} \Phi^k_{ij}.
 $$
 \item Hessian of $f$
 $$
 \nabla^2_h f_{ij} = f_{ij} - \frac{1}{2} \Phi^k_{ij} f_k.
 $$
\item  Riemann tensor
 $$
\mbox{\rm{Riem}}_{ikjl} = \frac{1}{4} (\Phi_{ila} \Phi^{a}_{kj} - \Phi_{ija} \Phi^{a}_{kl}).
$$
\item Ricci tensor
$$
\mbox{Ric}_{ij} = \frac{1}{4}  \Bigl( \Phi_{iab} \Phi^{ab}_j + \Phi_{ijk} \Phi^{k} \Bigr) = \frac{1}{4} (g_{ij} + \Phi_{ijk} \Phi^k ).
$$
\item Bakry-Emery tensor
$$
(\mbox{Ric}_{\mu})_{i,j}  =  \mbox{Ric}_{ij}  + \nabla_h^2 P_{ij} =
\frac{1}{4} g_{ij} + \frac{1}{2} \Phi_{ij}.
$$
\end{enumerate}

Let us recall some details about  computations of the weighted Laplacian.
  We stress that in this section {\bf we omit the subscript $h$} for the sake of simplicity, i.e.
  the symbols $\nabla, \nabla^2$ etc. are always related to the Hessian metric $h$, but not to Eudlidean metric.

Let us recall that given a tensor $T$  its Laplacian is defined  as follows:
  $$
  \Delta T = \Phi^{pq} \nabla_p \nabla_q T.
  $$
   Here $\nabla_p T$ is the covariant derivative, which means, in particular, that
   $$
   \nabla_p \Phi_{ij}=0, \ \Delta \Phi_{ij}=0.
   $$
   Similarly, one can compute the weighted Laplacian
   $$
   LT = \Delta T - \frac{1}{2} \Phi^k \nabla_k T.
   $$

The following weighted Leplacians of several important tensors  are taken from \cite{KK2}.

\begin{lemma}
\label{lphi}
$$
L \Phi_i =  \frac{1}{2} \Phi_i + \frac{1}{4} g_{i}^k \Phi_k.
$$
\end{lemma}

     \begin{lemma}
     \label{Lphi3}
     $$
L \Phi_{iab}  =\frac{1}{2} \Phi_{iab}
- \frac{1}{2} \Phi^{l}_{ik} \Phi^{m}_{al} \Phi^{k}_{bm}
+ \frac{1}{4} \Bigl( g^{k}_{i} \Phi_{kab} +  g^{k}_{a} \Phi_{kib} +  g^{k}_{b} \Phi_{kia} \Bigr).
$$
\end{lemma}

\begin{lemma}
\label{lgij}
\begin{align*}
L g_{ij} =  g_{ij} +\frac{1}{2} g_{ki} g^{k}_{j} + 2 \nabla_p \Phi_{iab} \nabla^p  \Phi_{j}^{ab}
+    8 \mbox{\rm{Riem}}_{iabc} \mbox{\rm{Riem}}_{j}^{abc}.
\end{align*}
\end{lemma}

Let us specify some of these results in the two-dimensional case.

  We deal throughout with the  orthonormal frame $(n,v)$, where
 $$n = \frac{\nabla \Phi}{|\nabla \Phi|}.$$
 Therefore $\Phi_n = 1$ and $\Phi_v = 0$.
  We will often write $\Phi_{ijn} |\nabla \Phi|$ instead of   $\Phi_{ijk}\Phi^k$.

  Let us recall that in the two-dimensional case $$   \mbox{Ric}_{ij} = R \Phi_{ij},$$ where $R$ is the sectional curvature.
  In order to make the formulas less heavy  we will use throughout the following quantity instead of $R$:
  $$
  \lambda = 4 R.
     $$
     In particular, we have

\begin{equation}
4 \mbox{Ric}_{ij} =  g_{ij} + \Phi_{ijn} |\nabla \Phi| = \lambda \Phi_{ij}.
\label{eq_957} \end{equation}
  Applying this identity to couples of vectors $(n,n), (v,v), (n,v)$, one gets
  $$
  g(n,n) + \Phi_{nnn}|\nabla \Phi| =
  g(v,v) + \Phi_{vvn}|\nabla \Phi|
 = \lambda
   $$
   $$
     g(n,v) + \Phi_{vnn}|\nabla \Phi| =0.
   $$

   \begin{remark} (Bound for $\lambda$ from below).
   \label{below-bound}
Applying identities
$$
g(v,v)= \Phi_{vvv}^2 + 2 \Phi_{vvn}^2 + \Phi_{vnn}^2
$$
and
$$
g(v,v) + \Phi_{vvn}|\nabla \Phi| = \lambda
$$
we get
$$
\lambda \ge 2 \Phi_{vvn}^2 +  \Phi_{vvn}|\nabla \Phi| \ge - \frac{|\nabla \Phi|^2}{8}.
$$
      \end{remark}

  \begin{remark}
The two-dimensional  Riemann  tensor has a particular simple structure
  $$
 { \rm{Riem}}_{ijkl}
  = \frac{1}{4} (\Phi_{ila}\Phi^{a}_{kj} -  \Phi_{ika} \Phi_{jla}) = R (\Phi_{ik} \Phi_{j \ell} - \Phi_{i \ell} \Phi_{j k} ).
  $$
Hence
  $$
   { \rm{Riem}}_{nvnv} = \frac{1}{4} (\Phi_{nva}\Phi^{a}_{nv} -  \Phi_{nna} \Phi_{vva}) =  R =\frac{\lambda}{4}
  $$
  $$
  { \rm{Riem}}_{nvnv}  = -{ \rm{Riem}}_{vnnv} = - { \rm{Riem}}_{nvvn}
  = { \rm{Riem}}_{vnvn}  =  R= \frac{\lambda}{4}.
  $$
  Other components equal zero.

   In particular
   \begin{equation}
   \label{riem2}{\rm{Riem}}_{abcd} {\rm{Riem}}^{abcd} = 4 R^2 =\frac{\lambda^2}{4}.
   \end{equation}
   \end{remark}

   We list below other  important identities which will be applied in the computations.
   First, there are several {\bf trace identities } obtained by differentiating the K{\"a}hler-Einstein equation (see
   (\ref{eigenf})).

\begin{itemize}
\item[(T1)]
$$
\Phi_{abc} \Phi^{ab} = - \Phi_c
$$
\item[(T2)]
Taking the covariant derivative of (T1) one gets
   $$\nabla_a \Phi_{bcd} \Phi^{cd} = - (\nabla^2 \Phi)_{ab}$$
   \item[(T3)]
   Taking trace of $\lambda \Phi_{ij} = g_{ij} + \Phi_{ijk} \Phi^k$ one gets
   $$
   2 \lambda = {\rm Tr} g - |\nabla \Phi|^2 = \Phi_{abc} \Phi^{abc} - |\nabla \Phi|^2.
   $$
   \item[(T4)]    Differentiating the identity (T3) one gets
   \begin{equation}
   \label{1-ord}
   (\nabla_p \Phi_{abc}) \Phi^{abc} = (\nabla^2 \Phi)_{pk} \Phi^k + \nabla_p \lambda =  (\nabla^2 \Phi)_{pn} |\nabla \Phi| + \nabla_p \lambda.
   \end{equation}
\end{itemize}

Similar computations show:

   \begin{lemma}
\begin{itemize}
\item[(Tr)] Trace
$$
{\rm Tr} (\nabla^2 \Phi) = \Phi^{ij} \left( \Phi_{ij} - \frac{1}{2} \Phi_{ij}^k \Phi_k \right) = 2 + \frac{|\nabla \Phi|^2}{2}.
$$
\item[(HS)] Hilbert-Schmidt norm
$$
   \|\nabla^2 \Phi\|^2 = (\Phi_{ij} - \frac{1}{2} \Phi_{ij}^k \Phi_k) (\Phi^{ij} - \frac{1}{2} \Phi^{ij}_{\ell} \Phi^{\ell})
     = 2 +   |\nabla \Phi|^2  +  \frac{1}{4}g(n,n)|\nabla \Phi|^2.
$$
   \item[(D)] Determinant
    $$
    \det \nabla^2 \Phi = \frac{1}{2} \left[ {\rm Tr} (\nabla^2 \Phi) \right] ^2 - \frac{1}{2} \| \nabla^2 \Phi \|^2   =
     1 + \frac{|\nabla \Phi|^2}{2} + \frac{|\nabla \Phi|^2}{8} \bigl( |\nabla \Phi|^2 - g(n,n) \bigr).
     $$
\end{itemize}
\end{lemma}

         Finally, the following results for Laplacians follow from Lemma \ref{lphi} and Lemma \ref{lgij}.
   \begin{itemize}
   \item[(L1)]
   \begin{align*}
     L |\nabla \Phi|^2 & = L( \Phi^i \Phi_i) = 2 L(\Phi_i) \Phi^i + 2 \nabla_p \Phi_i \nabla^p \Phi^i = 2 L(\Phi_i) \Phi^i + 2  \| \nabla^2 \Phi \|^2
     \\ & =   |\nabla \Phi|^2 +  \frac{1}{2} g(n,n)|\nabla \Phi|^2  +2 \|\nabla^2 \Phi\|^2
     = 4 + 3  |\nabla \Phi|^2  +  g(n,n)|\nabla \Phi|^2.
   \end{align*}

   \item[(L2)]
   $$
L({\rm Tr} g) = {\rm Tr} g + \frac{1}{2}  \| g \|^2 + 8 {\rm{Riem}}_{abcd} {\rm{Riem}}^{abcd} + 2 \nabla_p \Phi_{abc} \nabla^p \Phi^{abc},
   $$
   where $\| g \|$ is the Hilbert-Schmidt norm of the tensor $g$.
   \end{itemize}

  \section{Main result}
In this section
$$
n=2.
$$

    \begin{theorem}
    \label{mainth}
    Let $\lambda = 4 R$, where $R$ is the sectional curvature of the Hessian metric $D^2 \Phi$.
    Then
  \begin{multline*}
    2 L (\lambda )
=(3 \lambda-1)(\lambda+1) + \frac{|\nabla \Phi|^2}{8 \lambda + |\nabla \Phi|^2}  (3\lambda -1)^2
+ \frac{ 16|\nabla \lambda|^2
+ 2\bigl[4 (1- \lambda) +  |\nabla \Phi|^2 \bigr] \langle \nabla \Phi, \nabla \lambda \rangle}{8 \lambda + |\nabla \Phi|^2} .
    \end{multline*}
     \end{theorem}

We postpone the proof of Theorem \ref{mainth} for now, and present its following corollary:

 \begin{corollary}
The sectional curvature function $R$ satisfies
    $$
    R \le \frac{1}{12}.
    $$
      \end{corollary}

      \begin{proof} The idea is to apply Theorem \ref{mainth} and the maximum principle. We proceed as follows:

    {\bf Step 1.}       Assume first that  for any $\varepsilon>0$
    \begin{equation}
          \limsup_{x \to \infty} \lambda(x)  e^{-\varepsilon \Phi(x)} \le 0.
           \label{eq_300}
    \end{equation}
           In particular, $\lambda e^{-\varepsilon \Phi}$ has a local maximum for every $\varepsilon>0$.
        Applying the maximum principle to this function we get at the maximum point
        $$
         \frac{\nabla \lambda}{\lambda} = \varepsilon  \nabla \Phi,
           $$
           $$
           \frac{L \lambda}{\lambda} - \Bigl| \frac{\nabla \lambda}{\lambda}\Bigr|^2 \le 2 \varepsilon
           $$
     (we use here the relation $L \Phi=2$).
       By theorem \ref{mainth}
       $$
           2 L (\lambda )
=\lambda (3 \lambda-1) \Bigl[1  +  \frac{3 |\nabla \Phi|^2 + 8}{8\lambda + |\nabla \Phi|^2} \Bigr]
+ \frac{ 16|\nabla \lambda|^2
+ 2\bigl[4 (1- \lambda) +  |\nabla \Phi|^2 \bigr] \langle \nabla \Phi, \nabla \lambda \rangle}{8 \lambda + |\nabla \Phi|^2} .
    $$
    Substituting the above relations into this formula one gets that at any local maximum point the following inequality holds
       $$
           (3 \lambda-1) \Bigl[\frac{1}{2} +  \frac{3 |\nabla \Phi|^2 + 8}{16\lambda +2 |\nabla \Phi|^2} \Bigr]
+ \frac{ 16 \varepsilon^2 \lambda |\nabla \Phi|^2
+ 2 \varepsilon \bigl[4 (1- \lambda) +  |\nabla \Phi|^2 \bigr]  |\nabla \Phi|^2}{16 \lambda + 2|\nabla \Phi|^2} \le 2 \varepsilon + \varepsilon^2 |\nabla \Phi|^2.
    $$
    Rearranging the terms one gets
    $$
     (3 \lambda-1) \Bigl[\frac{1}{2}  +  \frac{3 |\nabla \Phi|^2 + 8}{16\lambda + 2|\nabla \Phi|^2} \Bigr]
+ \frac{
 8 \varepsilon (1- \lambda) +  2\varepsilon (1- \varepsilon) |\nabla \Phi|^4}{16 \lambda + 2|\nabla \Phi|^2} \le 2 \varepsilon.
    $$

Equivalently,
\begin{equation}    \frac{3 \lambda - 1}{2}  - 2 \varepsilon  + \frac{ [3 (3 \lambda -1) + 2 \varepsilon (1 - \varepsilon)] |\nabla \Phi|^2 + 8 [ (3 - \varepsilon) \lambda-1] + 8 \varepsilon  }{16 \lambda + 2 |\nabla \Phi|^2} \le 0. \label{eq_515} \end{equation}
We may assume that $\varepsilon \le  1$, as later $\varepsilon$ would tend to zero. We thus deduce from
(\ref{eq_515}) that
$$
 \frac{3 \lambda - 1 - 4 \varepsilon}{2}  + \frac{ 3 (3 \lambda -1)  |\nabla \Phi|^2 + 8 [ (3 - \varepsilon) \lambda-1] + 8 \varepsilon  }{16 \lambda + 2 |\nabla \Phi|^2} \le 0. $$
  It is therefore impossible that both $3 \lambda - 1 \geq 4 \varepsilon$ and $(3 - \varepsilon) \lambda \geq 1$.
    Hence
    $$
    \lambda \le \max \Bigl( \frac{1+ 4 \varepsilon}{3},  \frac{1}{3 - \varepsilon} \Bigr),
    $$
     at any local maximum of $\lambda e^{-\varepsilon \Phi}$. By the maximum principle, we have that everywhere in $\RR^n$,
    \begin{equation}
    \label{eq_529}
    \lambda e^{-\varepsilon  \Phi} \le \max \Bigl( \frac{1+ 4 \varepsilon}{3},  \frac{1}{3 - \varepsilon} \Bigr) e^{- \varepsilon \min \Phi},
    \end{equation}
    Tending $\varepsilon$ to zero we get $\lambda \le \frac{1}{3}$, hence $R \le \frac{1}{12}$.

\medskip    {\bf Step 2. Rational polytopes.}

    Assume that the convex body $K \subseteq \RR^n$ is a rational polytope, in the sense that all  the coordinates  of all of the vertices of $K$
     are rational numbers. It is known that in this case,
    \begin{equation}  \sup_{x \in \RR^n} |\lambda(x)| < \infty, \label{eq_1003}
    \end{equation}
    and in particular (\ref{eq_300}) holds true. In fact, in this case, one may consider the Riemannian metric on
    the complex torus $\CC^n / (2 \pi \sqrt{-1} \cdot \ZZ^n)$ given by
    \begin{equation}  \tilde{h}(x,y) = \Phi_{ij}(x) d x^i dx^j + \Phi_{ij}(x) d y^i d y^j \label{eq_1027}
    \end{equation}
    where $(x,y) = (x^1,\ldots,x^n, y^1,\ldots, y^n) \in \RR^n \times \RR^n \cong \CC^n$ are the standard coordinates, and where the $1$-forms $dx^1,\ldots, dx^n, d y^1,\ldots, dy^n$
    are well-defined on $\CC^n$ and also on the quotient $\CC^n / (2 \pi \sqrt{-1} \cdot \ZZ^n)$. The Riemannian metric $\tilde{h}$ is in fact a K\"ahler metric on the complex manifold
    $\CC^n / (2 \pi \sqrt{-1} \cdot \ZZ^n)$.  The equation (\ref{KE}) is equivalent to the assertion that the metric $\tilde{h}$ is an Einstein metric. Thus, up to translation, $\Phi$ is the unique function on $\RR^n$ satisfying (\ref{eq_1026}) for which the metric $\tilde{h}$ as defined in (\ref{eq_1027}) is an Einstein metric.

    \smallskip This toric K\"ahler-Einstein metric $\tilde{h}$
    was studied extensively in the complex geometry literature. We refer the reader to Legendre \cite{Legendre} and
    references therein for the following non-trivial fact:
    When $K$ is a rational polytope, the metric $\tilde{h}$ admits a completion which is a smooth, compact orbifold. In particular all sectional curvatures of $\tilde{h}$
    are uniformly bounded on $\CC^n / (2 \pi \sqrt{-1} \cdot \ZZ^n)$. Since $R = \lambda/4$ is a sectional curvature of $\tilde{h}$, the proof of (\ref{eq_1003})
    is complete.

\medskip
     {\bf Step 3. Approximation by rational polytopes.}
     Assume that $K \subseteq \RR^n$ is a convex body with barycenter at the origin. Then there
     exist rational polytopes $K_1,K_2,\ldots$ with barycenters at the origin such that $K_m \longrightarrow K$
     in the Hausdorff metric. By solving (\ref{KE}) for $K_m$ we obtain a sequence of convex functions $\Phi_m$, and
     according to \cite[Proposition 2.1]{K_part_I} we may assume that
 \begin{equation}  \Phi_m \longrightarrow \Phi \label{eq_1130}
 \end{equation}
     locally uniformly in $\RR^n$. We claim that
 \begin{equation}  D^2 \Phi_m \longrightarrow D^2 \Phi \label{eq_1131}
 \end{equation}
     locally uniformly in $\RR^n$. Indeed, it follows from (\ref{eq_1130}) that if a subsequence of $D^2 \Phi_m$
     converges locally uniformly to a certain limit, then this limit must equal $D^2 \Phi$. Thus it suffices to show that the sequence $D^2 \Phi_m$
     is precompact in local uniform norm. This follows from the Arzela-Ascoli theorem and the local uniform
     bound on the third derivatives (see Corollary \ref{three-est}). Thus (\ref{eq_1131}) is proven.

     \smallskip For a point $x \in \RR^n$
     and $r > 0$ write $B_m(x, r)$ for the collection of all points whose Riemannian distance from $x$, with respect to the metric induced by $D^2 \Phi_m$,
     is at most $r$. We set $B(x,r)$ for the corresponding ball with respect to the metric induced by $D^2 \Phi$.
     It follows from (\ref{eq_1131}) that for any $x \in \RR^n$ and $r_0 > 0$ for which $B(x, r_0) \subseteq \RR^n$ is compact,
     \begin{equation}  \int_{B_m(x,r_0)} \sqrt{ \det D^2 \Phi_m} \ dx \stackrel{m \rightarrow \infty}\longrightarrow \int_{B(x,r_0)} \sqrt{ \det D^2 \Phi} \ dx.
     \label{eq_1158} \end{equation}
     Thanks to the previous step, we know that the sectional curvature $R_m$ of the metric induced by $D^2 \Phi_m$ satisfies
     $R_m(x) \leq 1/12$ for any $x \in \RR^n$ and $m \geq 1$. Fix $x \in \RR^n$ and $r_0 > 0$ such that $B(x,r_0)$ is compact
     with a unique geodesic connecting any point $y \in B(x, r_0)$ to $x$. Then for a sufficiently large $m$, also $B_m(x,r_0)$
    is compact with unique geodesics to the center $x$,  with respect to the metric induced by $D^2 \Phi_m$.
By the Rauch comparison theorem (e.g., \cite[Section 6.5]{BBI}), for any $0 < r < r_0$ there is a metric-contraction from the Riemannian ball $B_m(x, r)$ onto the geodesic ball of radius $r$ in a two-dimensional sphere of radius $\sqrt{12}$.
    Consequently, we may compare areas and obtain the inequality
     $$ \int_{B_m(x,r)} \sqrt{ \det D^2 \Phi_m} \ dx \geq 24 \pi \cdot (1 - \cos(r / \sqrt{12}) ), $$
     where the right-hand side is the area of the geodesic ball in the sphere. It thus follows from (\ref{eq_1158}) that for all $0 < r < r_0$,
     \begin{equation}  \frac{12}{\pi r^4} \left( \pi r^2 - \int_{B(x,r)} \sqrt{ \det D^2 \Phi} \right) dx \leq \frac{12}{\pi r^4} \left( \pi r^2 - 24 \pi \cdot (1 - \cos(r / \sqrt{12}) ) \right). \label{eq_1218} \end{equation}
     When $r \rightarrow 0^+$, the right-hand side of (\ref{eq_1218}) tends to $1/12$, while the left hand-side tends to the sectional curvature $R(x)$.
     Therefore $R \leq 1/12$.
      \end{proof}

      \begin{remark}
      The fact that $R = \frac{1}{4} (\Phi_{abc} \Phi^{abc} - |\nabla \Phi|^2)$ is bounded from above
      looks highly non-trivial. It can be shown that the quantity   $\Phi_{abc} \Phi^{abc}$ tends to infinity
      and, moreover, has exponential growth. We show in the last section that in general $R$ is not bounded from below.
      \end{remark}

                    \begin{corollary}
        The function $\Phi$ is uniformly convex in Riemannian metric, more precisely
        $$\nabla^2 \Phi \ge \frac{5}{6}.
        $$
    In particular, the manifold $(\mathbb{R}^n, D^2 \Phi)$
    is geodesically convex.
            \end{corollary}

\begin{proof}
$$
        (\nabla^2 \Phi)_{ij}   = \Phi_{ij}  - \frac{1}{2} \Phi_{ijk}\Phi^k =
\frac{1}{2} g_{ij} + \Phi_{ij}  - \frac{1}{2}\bigl( g_{ij} + \Phi_{ijk}\Phi^k  \bigr) =
\frac{1}{2} g_{ij} + \bigl(1 - \frac{\lambda}{2}\bigr) \Phi_{ij}  \ge \frac{5}{6} \Phi_{ij}.
$$
Thus in the manifold $(\mathbb{R}^n, D^2 \Phi)$ we have a global convex function $\Phi$ satisfying
$\lim_{x \rightarrow \infty} \Phi(x) = +\infty$. This is known to imply that the manifold is geodesically convex, that any two points have at least one geodesic connecting them.
 \end{proof}

We continue with the proof of Theorem \ref{mainth} and the required lemmata. The argument
involves linear algebra computations that are based on the differentiations described in the
previous section. We fix a point $x_0$ in $\RR^n$, and our goal is to
prove the formula for $2 L(\lambda)$ from Theorem \ref{mainth} at this point $x_0$.
The smooth function $\Phi$ is strongly convex, as $\det D^2 \Phi$ never vanishes. In particular,
the differential of $\Phi$ can vanish at most at one point. By continuity, in proving
Theorem \ref{mainth} we may assume that $\nabla \Phi(x_0) \neq 0$.

\smallskip We proceed with algebraic computations at the point $x_0$.
Let $(e,u)$ be an two tangent vectors at the point $x_0$, which constitute an orthonormal frame consisting of eigenvectors of $\nabla^2 \Phi$.
Let
$$
\Lambda(e), \Lambda(u)
$$
be the corresponding eigenvalues.
In particular,
$$
0 = \bigl(\nabla^2 \Phi\bigr)_{eu} = \Phi_{eu} - \frac{1}{2} \Phi_{eun} \Phi_n.
$$
Since $\Phi_{eu} = 0$ and $\Phi_n \neq 0$ then
$$
\Phi_{eun}=0.
$$

\begin{remark}
The tensor
      $$
      Q_{abcd} = \nabla_{a} \Phi_{bcd}
      $$
      is symmetric with respect to any permutation of coordinates.
      \end{remark}

     From (T2) we infer that the components of the tensor $Q$ do satisfy the following linear equations.
      \begin{equation}
      \label{Q1}
      Q_{eeee} + Q_{eeuu} = - \bigl(\nabla^2 \Phi\bigr)_{ee}  = - \Lambda(e)
      \end{equation}
      \begin{equation}
      \label{Q2}
      Q_{eeeu} + Q_{euuu} = -\bigl(\nabla^2 \Phi\bigr)_{eu} = 0
      \end{equation}
          \begin{equation}
          \label{Q3}
      Q_{eeuu} + Q_{uuuu} = -\bigl(\nabla^2 \Phi\bigr)_{uu} = - \Lambda(u)
      \end{equation}

   \begin{lemma}
   \label{llambda} At the point $x_0$, we have
   $$
    2 L (\lambda ) =
    2  \lambda + 3 \lambda^2 + \lambda |\nabla \Phi|^2 + 16 \bigl( Q^2_{eeuu} + Q^2_{eeeu}\bigr) + 2 Q_{eeuu} (4 + |\nabla \Phi|^2)
    $$
    \end{lemma}
   \begin{proof}
   By (\ref{eq_957}) and  (T1), one gets
   \begin{equation}
   \label{trg2}
  \| g \|^2 =  \bigl \| \lambda \Phi_{ij} - \Phi_{ijn} |\nabla \Phi| \bigr \|^2 = 2 \lambda^2 + 2\lambda |\nabla \Phi|^2 + g(n,n)|\nabla \Phi|^2,
   \end{equation}
where $\| g \|^2$ is the square of the Hilbert-Schmidt norm of the tensor $g$.
   Then it follows from (L2), (\ref{riem2}) and  (\ref{trg2}).
   $$
   L({\rm Tr} g) =  {\rm Tr}(g) + 3 \lambda^2 + \lambda |\nabla \Phi|^2 + \frac{1}{2} g(n,n)|\nabla \Phi|^2 +  2 \nabla_p \Phi_{abc} \nabla^p \Phi^{abc}.
   $$
Then the identities    $2 L (\lambda ) = L(   {\rm Tr} g -  |\nabla \Phi|^2) $ (follows from (T3))  and  (L1) imply
\begin{equation}
\label{190517}
    2 L (\lambda ) =
   -4 + 2  \lambda + 3 \lambda^2 + (\lambda-2) |\nabla \Phi|^2 - \frac{1}{2} g(n,n)|\nabla \Phi|^2 +
   2 \nabla_p \Phi_{abc} \nabla^p \Phi^{abc}.
    \end{equation}
   Using (\ref{Q1})-(\ref{Q3})   we rewrite $
                                \nabla_p \Phi_{abc} \nabla^p \Phi^{abc}$ in terms of $Q$:
                                \begin{align*}
                                \nabla_p \Phi_{abc} \nabla^p \Phi^{abc} & = Q^2_{eeee} + 4 Q^2_{eeeu} + 6 Q^2_{eeuu} + 4 Q^2_{euuu} + Q^2_{uuuu}
                                \\& = ( Q_{eeuu} +  \Lambda(e))^2 + 8 Q^2_{eeeu} + 6 Q^2_{eeuu} +
                                                                ( Q_{eeuu} + \Lambda(u) )^2
                                \\& = 8 (Q^2_{eeuu} +    Q^2_{eeeu}) + 2 Q_{eeuu} {\rm Tr}(\nabla^2 \Phi)
                                 + \| \nabla^2 \Phi\|^2
                                \\&
                                = 8 (Q^2_{eeuu} +    Q^2_{eeeu}) +  Q_{eeuu} \Bigl( 4 + {|\nabla \Phi|^2}\Bigr)
                                +  2 + |\nabla \Phi|^2 + \frac{1}{4} g(n,n) |\nabla \Phi|^2.
                                \end{align*}
 Substituting the above formula into (\ref{190517}) we get the result.
   \end{proof}

   Our next goal is to rewrite the result of Lemma \ref{llambda}
   in terms of $\lambda, \nabla \lambda$.
   Note that
   $$
   (\nabla^2 \Phi)_{en} |\nabla \Phi| = (\nabla^2 \Phi)_{ee} \Phi_e + (\nabla^2 \Phi)_{eu} \Phi_u = \Lambda(e) \Phi_e.
   $$
   Analogously
    $$
   (\nabla^2 \Phi)_{un} |\nabla \Phi| =  \Lambda(u) \Phi_u.
   $$
   From (T4) we thus obtain  the following Lemma.
      \begin{lemma}
      The components of the tensor $Q$  satisfy the following  equations:
      \begin{equation}
      \label{phien}
      Q_{eeee} \Phi_{eee} + 3 Q_{eeeu} \Phi_{eeu} + 3Q_{eeuu} \Phi_{euu} +    Q_{euuu} \Phi_{uuu}
      =
      \Lambda(e) \Phi_e + \lambda_e
      \end{equation}
       \begin{equation}
       \label{phiun}
      Q_{eeeu} \Phi_{eee} + 3 Q_{eeuu} \Phi_{eeu} + 3Q_{euuu} \Phi_{euu} +    Q_{uuuu} \Phi_{uuu}
      = \Lambda(u) \Phi_u+ \lambda_u.
      \end{equation}
      \end{lemma}

     Let us consider equations  (\ref{Q1})-(\ref{Q3})  and (\ref{phien})- (\ref{phiun})
     as a system of five linear equations with five unknown variables which are the components of the symmetric tensor $Q$. This is an important feature of dimension $n = 2$, in higher dimensions the number of unknown variables in this approach seems to exceed the number of linear equations. This system is solved in the next lemma.
     From (T1) we obtain
     $$ \Phi_{eee} + \Phi_{euu} = - \Phi_e, \qquad \Phi_{uuu} + \Phi_{eeu} = -\Phi_u. $$

      \begin{lemma}
      \label{QeeuuQeeeu}
      One has         $$
                             Q_{eeuu} = \frac{   AC +  BD }{C^2 + D^2},
                             Q_{eeeu} = \frac{   A D  -     BC }{C^2 + D^2},
                             $$
                             and hence
                             $$ Q_{eeuu}^2 + Q_{eeeu}^2 = \frac{A^2+B^2}{C^2 + D^2}, $$
                                                         where
      $$
       A = - \Phi_{euu}\Lambda(e)   + \lambda_e,
       $$
       $$
       B = -\Phi_{uee} \Lambda(u)  + \lambda_u,
       $$
         $$
         C= 3 \Phi_{euu} - \Phi_{eee} = 4 \Phi_{euu} + \Phi_e,
         $$
         $$
         D = 3  \Phi_{eeu}  - \Phi_{uuu} = 4 \Phi_{uee} + \Phi_u.
         $$
                    \end{lemma}
      \begin{proof}
      Substituting   $Q_{eeee} = - \Lambda(e)-  Q_{eeuu}$, $ Q_{uuuu} = - \Lambda(u) -  Q_{eeuu}$, $Q_{euuu} = - Q_{eeeu}$  into (\ref{phien}), (\ref{phiun}) one gets two linear equations:
              $$
              C Q_{eeuu} + D Q_{eeeu} = \Lambda(e) \Phi_e + \Lambda(e) \Phi_{eee}  + \lambda_e = - \Phi_{euu} \Lambda(e) + \lambda_e,
              $$

      $$
       D Q_{eeuu}     -   C Q_{eeeu} = \Lambda(u) \Phi_{uuu} + \Lambda(u) \Phi_u + \lambda_u  =  - \Phi_{uee} \Lambda(u) +  \lambda_u.
      $$
      We rewrite these two real equation as a single complex-linear equation
      $$ (C - i D) (Q_{eeuu} + i Q_{eeeu}) = A - i B. $$
      Solving this equation one gets the desired result.
         \end{proof}c

 \begin{lemma}
 \label{17.03}

 \begin{enumerate}
 \item
 For {\bf any} orthonormal frame $(e,u)$
  $$
  \lambda = 2 ( \Phi_{euu}^2 + \Phi_{eeu}^2) + \Phi_e \Phi_{uue} + \Phi_u \Phi_{eeu}
 $$
 \item
 For the frame $(e,u)$ of eigenvectors of $\nabla^2 \Phi$,
 $$
\lambda =  \frac{2\Phi_{euu}}{\Phi_e} \bigl( \Lambda(e) - \Lambda(u) \bigr) = \frac{2\Phi_{uee}}{\Phi_u} \bigl( \Lambda(u) - \Lambda(e) \bigr)
$$
\item  For the frame $(e,u)$ of eigenvectors of $\nabla^2 \Phi$,
$$
2  \bigl( \lambda_e \Phi_{euu} -  \lambda_u \Phi_{uee} \bigr) (\Lambda(u) - \Lambda(e))
= -  \lambda \langle \nabla \Phi, \nabla \lambda \rangle.
$$

 \end{enumerate}
  \end{lemma}
  \begin{proof}
  1) Follows from the trace identity $-\Phi_e  = \Phi_{eab} \Phi^{ab} = \Phi_{eee} + \Phi_{euu}$, the relations
 $$
g(e,e) = \Phi_{eee}^2 + 2  \Phi_{eeu}^2 +  \Phi_{euu}^2  =
 (\Phi_{euu} + \Phi_e)^2 + 2  \Phi_{eeu}^2 +  \Phi_{euu}^2   = 2 ( \Phi_{euu}^2 + \Phi_{eeu}^2) + 2 \Phi_e \Phi_{euu} + \Phi^2_e,
$$
$$
g(u,u)                                                                                                                                                                                                                                                                                                                                                                                                                                                                                                                   = 2 ( \Phi_{euu}^2 + \Phi_{eeu}^2) + 2 \Phi_u \Phi_{eeu} + \Phi^2_u,
 $$
 and $g(e,e) + g(u,u) = 2 \lambda + |\nabla \Phi|^2$, according to (T3).

2) From the equation $
\Phi_{eun}=0
$ we  infer
$0 = \Phi_{e} \Phi_{eeu} + \Phi_{u} \Phi_{euu}$. Hence
\begin{equation}
\frac{\Phi_{euu}}{\Phi_e} =
-\frac{\Phi_{uee}}{\Phi_u}.
\label{eq_1111}
\end{equation}
We note that by (T1),
$$
\Lambda(e) = (\nabla^2 \Phi)_{ee} = 1 - \frac{1}{2} \Phi_e \Phi_{eee} -  \frac{1}{2} \Phi_u \Phi_{eeu}  = 1 + \frac{\Phi^2_e}{2}
+ \frac{1}{2} \bigl(\Phi_e\Phi_{euu}  - \Phi_u \Phi_{eeu} \bigr).
$$
$$
\Lambda(u)   = 1 + \frac{\Phi^2_u}{2}
- \frac{1}{2} \bigl(\Phi_e\Phi_{euu}  - \Phi_u \Phi_{eeu} \bigr).
$$
Hence
\begin{equation}
\Lambda(e) - \Lambda(u) =  \frac{1}{2} \bigl( \Phi^2_e - \Phi^2_u\bigr)
+ \bigl(\Phi_e\Phi_{euu}  - \Phi_u \Phi_{eeu} \bigr).
\label{eq_1046} \end{equation}
Applying  relation $
\frac{\Phi_{euu}}{\Phi_e} =
-\frac{\Phi_{uee}}{\Phi_u}
$ and formula (1)
one gets
\begin{align*}
  \lambda &= 2 ( \Phi_{euu}^2 + \Phi_{eeu}^2) + \Phi_e \Phi_{uue} + \Phi_u \Phi_{eeu}
   = 2 ( \Phi_{euu}^2 + \Phi_{eeu}^2) +  \frac{\Phi_{euu}}{\Phi_e} \Bigl( \Phi_e^2 - \Phi^2_u \Bigr)
  \\&
  =  \frac{2\Phi_{euu}}{\Phi_e}  \Bigl( \frac{\Phi^2_e - \Phi^2_u}{2} + \frac{\Phi_e}{\Phi_{euu}} ( \Phi_{euu}^2 + \Phi_{eeu}^2)\Bigr)
  = \frac{2\Phi_{euu}}{\Phi_e}  \Bigl( \frac{\Phi^2_e - \Phi^2_u}{2} + {\Phi_e}{\Phi_{euu}}  + \frac{\Phi_e}{\Phi_{euu}} \Phi_{eeu}^2 )\Bigr)
  \\& = \frac{2\Phi_{euu}}{\Phi_e} \Bigl(  \frac{1}{2} \bigl( \Phi^2_e - \Phi^2_u\bigr)
+ \bigl(\Phi_e\Phi_{euu}  - \Phi_u \Phi_{eeu} \bigr) \Bigr) =  \frac{2\Phi_{euu}}{\Phi_e} \bigl( \Lambda(e) - \Lambda(u) \bigr).
\end{align*}

3) In view of (2), it suffices to prove that
$$ 2 (\lambda_e \Phi_{euu} - \lambda_u \Phi_{uee}) (\Lambda(u) - \Lambda(e)) = - 2 \frac{\Phi_{euu}}{\Phi_e} \bigl( \Lambda(e) - \Lambda(u) \bigr) (\Phi_e \lambda_e + \Phi_u \lambda_u). $$
This is  equivalent to
$$ \lambda_e \Phi_{euu} - \lambda_u \Phi_{uee} = (\lambda_e \Phi_e + \lambda_u \Phi_u) \frac{\Phi_{euu}}{\Phi_e}
 $$
which holds true in view of (\ref{eq_1111}).
\end{proof}

The next lemma follows immediately from Lemma
      \ref{QeeuuQeeeu} and Lemma \ref{17.03}(1).

\begin{lemma}
\label{QABCD}
$$
16(Q^2_{eeuu} +    Q^2_{eeeu}) +  2   Q_{eeuu} (4 + |\nabla \Phi|^2) =
\frac{ 16(A^2 + B^2) + 2 ( 4 + {|\nabla \Phi|^2}) ( AC + BD) }{C^2 + D^2},
$$
$$
C^2 + D^2  =  (4 \Phi_{euu} + \Phi_e)^2  + (4 \Phi_{uee} + \Phi_u)^2 =    8 \lambda + |\nabla \Phi|^2.
$$
\end{lemma}

    \begin{lemma}
    \label{longexp}
    The following identity holds
    \begin{eqnarray*} & 16(  A^2   + B^2 ) + 2( 4 + {|\nabla \Phi|^2}) (AC+BD)
    \\&
   =    \lambda \Bigl[  {|\nabla \Phi|^2} \bigl( \lambda - |\nabla \Phi|^2\bigr)  - 6|\nabla \Phi|^2 -8 \Bigr]
 + 16|\nabla \lambda|^2
+ 2\bigl[4 (1- \lambda) +  |\nabla \Phi|^2 \bigr] \langle \nabla \Phi, \nabla \lambda \rangle.
 \end{eqnarray*}
    \end{lemma}
    \begin{proof}
    Using identities
 $
  A=     - \Lambda(e) \Phi_{euu}  + \lambda_e,
  $
 $
  B=   - \Lambda(u) \Phi_{eeu} + \lambda_u
  $
  one gets
  $$
  A^2 + B^2 = \Lambda^2(e) \Phi^2_{euu} + \Lambda^2(u) \Phi^2_{eeu} + |\nabla \lambda|^2
  - 2 (\lambda_e \Lambda(e) \Phi_{euu}  + \lambda_u  \Lambda(u) \Phi_{eeu}  ).
  $$
  \begin{align*}
  AC + BD & = [- \Lambda(e) \Phi_{euu} + \lambda_e]  (4  \Phi_{euu} + \Phi_e) + [ - \Lambda(u) \Phi_{eeu} + \lambda_u]  (4  \Phi_{eeu} + \Phi_u)
 \\& = - 4 \Lambda(e) \Phi_{euu}^2 - 4 \Lambda(u) \Phi_{eeu}^2 - \Lambda(e) \Phi_e \Phi_{euu} - \Lambda(u) \Phi_u \Phi_{eeu}
 \\&
 + \lambda_e (4  \Phi_{euu} + \Phi_e) +  \lambda_u  (4  \Phi_{eeu} + \Phi_u).
  \end{align*}
  From the trace identity $\Phi^{ab} (\nabla^2 \Phi)_{ab} = 2 + \frac{|\nabla \Phi|^2}{2}$ we infer
 $$
 4 + {|\nabla \Phi|^2} = 2 (\Lambda(e) + \Lambda(u)).
 $$
Thus,
 \begin{eqnarray*}
& 16(  A^2   + B^2 ) + 2( 4 + {|\nabla \Phi|^2}) (AC+BD)
\\& = 16( \Lambda^2(e) \Phi^2_{euu} + \Lambda^2(u) \Phi^2_{eeu})
+ 16|\nabla \lambda|^2
  - 32 (\lambda_e \Lambda(e) \Phi_{euu}  + \lambda_u  \Lambda(u) \Phi_{eeu}  ).
\\& - 4 (\Lambda(e) + \Lambda(u)) ( 4 \Lambda(e) \Phi_{euu}^2 + 4 \Lambda(u) \Phi_{eeu}^2 + \Lambda(e) \Phi_e \Phi_{euu} + \Lambda(u) \Phi_u \Phi_{eeu})
\\& + 4 ( \Lambda(e) + \Lambda(u)) \Bigl( \lambda_e (4  \Phi_{euu} + \Phi_e) +  \lambda_u  (4  \Phi_{eeu} + \Phi_u) \Bigr)
\\& = - 8 \Lambda(e) \Lambda(u) (\lambda - \Phi_e \Phi_{euu} - \Phi_u \Phi_{eeu}) - 4 (\Lambda(e) + \Lambda(u)) ( \Lambda(e) \Phi_e \Phi_{euu} + \Lambda(u) \Phi_u \Phi_{eeu})
\\& + 16|\nabla \lambda|^2
+ 2(4 + |\nabla \Phi|^2) \langle \nabla \Phi, \nabla \lambda \rangle
+ 16 \lambda_e \Phi_{euu} (\Lambda(u) - \Lambda(e)) + 16 \lambda_u \Phi_{uee} (\Lambda(e) - \Lambda(u))
\\& = - 8 \lambda \det \nabla^2 \Phi  + 4 \bigl( \Lambda(e) \Lambda(u)(\Phi_e \Phi_{euu} + \Phi_{u} \Phi_{uee}) - \Lambda^2(e) \Phi_{e} \Phi_{euu} - \Lambda^2(u) \Phi_{u} \Phi_{eeu}  \bigr)
\\&
+ 16|\nabla \lambda|^2
+ 2(4 + |\nabla \Phi|^2) \langle \nabla \Phi, \nabla \lambda \rangle
+ 16 \bigl( \lambda_e \Phi_{euu} -  \lambda_u \Phi_{uee} \bigr) (\Lambda(u) - \Lambda(e))
 \\&
    = - 8 \lambda \Bigl( 1 + \frac{|\nabla \Phi|^2}{2} + \frac{|\nabla \Phi|^2}{8} \bigl( |\nabla \Phi|^2 - g(n,n) \bigr) \Bigr)
+ 4 ( \Lambda(u) -  \Lambda(e))  (\Lambda(e) \Phi_e \Phi_{euu} - \Lambda(u) \Phi_{u} \Phi_{eeu})
\\&
+ 16|\nabla \lambda|^2
+ 2(4 + |\nabla \Phi|^2) \langle \nabla \Phi, \nabla \lambda \rangle
+ 16 \bigl( \lambda_e \Phi_{euu} -  \lambda_u \Phi_{uee} \bigr) (\Lambda(u) - \Lambda(e)).
 \end{eqnarray*}
 Next we note that by (\ref{eq_1046}),
 \begin{eqnarray*}
& (\Lambda(e) \Phi_e \Phi_{euu} - \Lambda(u) \Phi_{u} \Phi_{eeu})                                                                                                                                                                                                                                                                                                                                                                                                                                                                                                                  = \Lambda(e) \Phi_e \Phi_{euu} + \Lambda(u) \frac{\Phi^2_{u}}{\Phi_e} \Phi_{euu}
\\& = \frac{\Phi_{euu}}{\Phi_e} \bigl( \Lambda(e) \Phi^2_e +  \Lambda(u) \Phi^2_u\bigr)
= \frac{\Phi_{euu}}{\Phi_e}  \nabla^2 \Phi(\nabla \Phi, \nabla \Phi).
\end{eqnarray*}
Consequently,
$$
( \Lambda(u) -  \Lambda(e))  (\Lambda(e) \Phi_e \Phi_{euu} - \Lambda(u) \Phi_{u} \Phi_{eeu})
= ( \Lambda(u) -  \Lambda(e))   \frac{\Phi_{euu}}{\Phi_e}  \nabla^2 \Phi(\nabla \Phi, \nabla \Phi)
$$
 and  using Lemma \ref{17.03}(2) we get
$$
( \Lambda(u) -  \Lambda(e))  (\Lambda(e) \Phi_e \Phi_{euu} - \Lambda(u) \Phi_{u} \Phi_{eeu})
= -\frac{\lambda}{2} \nabla^2 \Phi(\nabla \Phi, \nabla \Phi).
$$
Finally, recall that by Lemma \ref{17.03}(3)
$$
16  \bigl( \lambda_e \Phi_{euu} -  \lambda_u \Phi_{uee} \bigr) (\Lambda(u) - \Lambda(e))
= - 8 \lambda \langle \nabla \Phi, \nabla \lambda \rangle.
 $$
 Hence
     \begin{eqnarray*} & 16(  A^2   + B^2 ) + 2( 4 + {|\nabla \Phi|^2}) (AC+BD)
    \\&
   =   - 8 \lambda \Bigl[ 1 + \frac{|\nabla \Phi|^2}{2} + \frac{|\nabla \Phi|^2}{8} \bigl( |\nabla \Phi|^2 - g(n,n) \bigr) + \frac{1}{4}|\nabla \Phi|^2  (\nabla^2 \Phi(n,n))  \Bigr]
\\& + 16|\nabla \lambda|^2
+ 2\bigl[4 (1- \lambda) +  |\nabla \Phi|^2 \bigr] \langle \nabla \Phi, \nabla \lambda \rangle.
 \end{eqnarray*}
 Finally, as $\nabla^2 \Phi(n,n) = 1 - |\nabla \Phi| \cdot \Phi_{nnn}  / 2$,
 \begin{eqnarray*}
  &1 + \frac{|\nabla \Phi|^2}{2} + \frac{|\nabla \Phi|^2}{8} \bigl( |\nabla \Phi|^2 - g(n,n) \bigr) + \frac{1}{4}|\nabla \Phi|^2  (\nabla^2 \Phi(n,n))
 \\& =
  1 + \frac{3|\nabla \Phi|^2}{4} + \frac{|\nabla \Phi|^2}{8} \bigl( |\nabla \Phi|^2 - g(n,n)  - \Phi_{nnn} |\nabla \Phi| \bigr)
  =
  1 + \frac{3|\nabla \Phi|^2}{4} + \frac{|\nabla \Phi|^2}{8} \bigl( |\nabla \Phi|^2 - \lambda \bigr).
 \end{eqnarray*}
Substituting this identity one gets the claim.
\end{proof}

\begin{proof}[Proof of Theorem \ref{mainth}] By  Lemma \ref{llambda} and Lemma \ref{QABCD},
$$ 2 L(\lambda) = 2 \lambda + 3 \lambda^2 + \lambda |\nabla \Phi|^2 + \frac{ 16(A^2 + B^2) + 2 ( 4 + {|\nabla \Phi|^2}) ( AC + BD) }{8 \lambda + |\nabla \Phi|^2} $$
We now plug in the formula from Lemma \ref{longexp}, and apply direct computations to complete the proof
of the theorem.
\end{proof}

  \section{Examples}

        It follows from (L2) above that $L(|\nabla \Phi|^2) > 0$ and in particular, $|\nabla \Phi|^2$ cannot be a constant function.
        From Theorem \ref{mainth}
        we thus conclude  that there are only two cases of the constant corvature in the two-dimensional case:
        $\lambda=0$ (cube) and $\lambda=\frac{1}{3}$ (simplex). In this section we consider both examples.
        In addition, we consider the ball and demonstrate that the Ricci tensor can be unbounded from below.

  \subsection{Cube}
  In the case of the cube
  $$
  K = [-1,1]^n
  $$
  the solution
  $$
  \Phi(x) = \sum_{i=1}^n \varphi(x_i)
  $$
  is a sum of one-dimensional independent potentials solving the following:
  $$
  \varphi''  = e^{-\varphi},
  $$
  $$ \varphi'(\mathbb{R}) = (-1,1).
  $$
  So $\varphi(t) = \log [2 \cosh^2 (t/2)]$.  Since the Riemannian manifold $(\mathbb{R}^n, D^2 \Phi)$ is Euclidean (it is in fact isometric to
  a cube of sidelength $\sqrt{2} \cdot \pi)$, one has $$\rm{Ric}=0.$$
  In particular, the following holds:
  $$
  \Phi_{abc} \Phi^{abc} = \sum_{i=1}^n \frac{(\varphi^{'''})^2}{(\varphi^{''})^3}(x_i)
  = \sum_{i=1}^n \frac{(\varphi')^2}{\varphi^{''}} (x_i) = |\nabla \Phi|^2
  = \sum_{i=1}^n e^{\varphi(x_i)} (\varphi')^2 (x_i).
  $$
  The latter expression has exponential growth. This means, in particular, that the natural Riemannian norms of  the first and third derivatives of the potential $\Phi$ are in general unbounded.

  \subsection{Ball}
  Here $K = B^n$ and $n \geq 2$. By uniqueness, the solution $\Phi$ of (\ref{KE}) with $\nabla \Phi(\RR^n) = B^n$ is (up to translation) of the form
  $$ \Phi(x) = \vphi(|x|) $$
  where the smooth function $\vphi$ on the half-line is increasing and it satisfies $\vphi'(0) = 0$ and $\vphi'(+\infty) = 1$. The equation (\ref{KE}) takes the form
 $$ \det \left[ \vphi'(|x|) \frac{Id}{|x|} + \left( \vphi''(|x|) - \vphi'(|x|) / |x| \right) \frac{x \otimes x}{|x|^2}     \right] = \exp(-\vphi(|x|)) $$
 or equivalently,
 $$ \left( \frac{\vphi'(r)}{r} \right)^{n-1} \cdot \vphi''(r) = e^{-\vphi(r)}. $$
 In the Riemannian manifold $(\mathbb{R}^n, D^2 \Phi)$, rays emanating from the origins are geodesics. Denote
 $$ D_n := \int_0^{\infty} \sqrt{\vphi''(r)} dr = \int_0^{\infty} e^{-\vphi(r)/2} \left( \frac{r}{\vphi'(r)} \right)^{\frac{n-1}{2}} dr < \infty $$
 since the integral clearly converges at $0$ (as $\vphi'(0) = 0$ and $\vphi$ is smooth) and at infinity (as the integrand decays exponentially).
 Then the Riemannian manifold $(\mathbb{R}^n, D^2 \Phi)$ is in fact an open Riemannian ball of a finite radius $D_n$ around $0$, denoted by $B(0, D_n)$.
 For  $0 < R < D_n$ let $t = t(R) > 0$ be such that
 $$ t = \int_0^{R} \sqrt{\vphi''(r)} dr. $$
Then $\lim_{R \rightarrow D_n} t(R) = +\infty$.
 Recall that the volume of the Euclidean unit sphere $S^{n-1}$ is $\kappa_n = n \pi^{n/2} / \Gamma(1 + n/2)$.
Then the Riemannian volume of the Riemannian sphere $\partial B(0, R)$ equals
$$ \kappa_n t^{n-1} e^{-\vphi(t) / 2} \cdot \frac{1}{\sqrt{ \vphi''(t)} } = \kappa_n ( t \vphi^{\prime}(t) )^{\frac{n-1}{2}} \longrightarrow +\infty $$
as $t \rightarrow \infty$ or as $R \rightarrow D_n$. Thus in this Riemannian manifold, a sequence of spheres of bounded radius have volumes tending to infinity.
This means that the Ricci curvature is not bounded from below.

\begin{remark} {\rm
What are the sharp estimates from below for the sectional curvature?
One bound is given in  Remark \ref{below-bound}. One might expect the two dimensional ball to be a natural example
to prove sharpness of this estimate. However, this is not the case. 

Indeed in polar coordinates, the metric is separated nicely as
$$ h = \vphi''(r) (dr)^2 + (r \vphi^{\prime}(r)) (d \theta)^2.$$
In the two-dimensional case, the sectional curvature is given by the usual formula
$$ H = -\frac{1}{2 \sqrt{ r \vphi'(r) \vphi''(r)}} \frac{d}{dr} \left[ \frac{\vphi'(r) + r \vphi''(r)}{\sqrt{ r \vphi'(r) \vphi''(r)}} \right]. $$
Since $\vphi'(r) \vphi''(r) = r e^{-\vphi(r)}$, we have
\begin{align*}
H = -\frac{e^{\vphi/2}}{2r} \cdot \frac{d}{dr} \left[ e^{\vphi/2} \frac{\vphi'}{r} + \frac{r}{\vphi'} e^{-\vphi/2} \right] = -\frac{(\vphi')^2}{4 r^2} e^{\vphi} + \frac{1}{4} + \frac{\vphi'}{2 r^3} - \frac{1}{r \vphi'} + \frac{r}{2 (\vphi')^3} e^{-\vphi}.
\end{align*}
Recall that as $r \rightarrow \infty$, the ratio $\vphi'(r) / r$ tends to one. Therefore, for large $r$,
$$ H = -\frac{(\vphi')^2}{4 r^2} e^{\vphi} + \frac{1}{4} + O \left( \frac{1}{r^2} \right). $$
On the other hand $|\nabla \Phi|^2 = (\vphi')^2 / \vphi'' = e^{\vphi} \cdot (\vphi')^3 / r $. We thus see that the estimate from Remark 2.4 is off by a factor of roughly $r^2/2$.}
\end{remark}

  \subsection{Simplex}

  This is the most important example for us, because this is the extremal situation, where the maximal value of the Ricci
  tensor is attained (at least in dimension 2).

  It can be easily checked (see explanations in \cite{K_moment}) that
  in case of the simplex
  $$
K = \bigl\{ \sum_{i=1}^{n} x_i \le 1, \ x_i \ge -1\bigr\}
$$
  the explicit solution is given by the formula
  $$
  \Phi = (n+1) \log \bigl(1+ \sum_{i=1}^n e^{x_i} \bigr) - \sum_{i=1}^n x_i + c(n)
  $$
  Differentiating twice one gets
  $$
  \Phi_i = (n+1) \frac{e^{x_i}}{1 + \sum_{i=1}^n e^{x_i}} - 1
  $$
  $$
  \Phi_{ij} = (\Phi_{i}+1)\delta_{ij}  - \frac{(1+\Phi_i)(1+ \Phi_j)}{n+1}
  $$
  The note that $D^2 \Phi$ is nothing else but the Fubini-Study metric (after appropriate complexification) in the complex projective space (see explanations in \cite{K_moment}).
  $$
  \Phi^{ij} = \frac{1}{\Phi_{i}+1} \delta_{ij}  - \frac{1}{(\sum_{i=1}^n \Phi_i)-1}.
  $$
  Next we proceed with the third derivatives
  $$
  \Phi_{ijk} = \Phi_{ik} \delta_{ij} - \frac{\Phi_{ik} (\Phi_j+1)}{n+1} - \frac{\Phi_{jk} (\Phi_i+1)}{n+1}.
  $$
  $$
  \Phi_{ij}^{k}  = \delta^k_{ij} - \frac{\delta_{i}^{k} (\Phi_j+1)}{n+1} - \frac{\delta_{j}^k (\Phi_i+1)}{n+1}.
  $$
  $$
  \Phi_{ijk} \Phi^k = \Phi_i \delta_{ij} - \frac{\Phi_i(\Phi_j+1) + \Phi_j(\Phi_i+1) }{n+1}.
  $$
  Here  $\delta^k_{ij}=1$ if $i=j=k$ and $\delta^k_{ij}=0$  in any other case.
  Let us compute tensor the $g_{ij}$
  \begin{align*}
  g_{ij} & = \Phi_{ik}^{l} \Phi_{jl}^{k} =
  \Bigl( \delta_{ik}^l - \frac{\delta_{i}^{l} (\Phi_k+1)}{n+1} - \frac{\delta_{k}^l (\Phi_i+1)}{n+1}\Bigr)
   \Bigl( \delta_{jl}^k - \frac{\delta_{j}^{k} (\Phi_l+1)}{n+1} - \frac{\delta_{l}^k (\Phi_j+1)}{n+1}\Bigr)
   \\& = \delta_{ij} \Bigl( 1- \frac{2(\Phi_i+1)}{n+1}\Bigr) + \frac{n+3}{(n+1)^2}
   (\Phi_i+1)(\Phi_{j}+1) - \frac{(\Phi_i+1)+(\Phi_{j}+1) }{n+1}.
  \end{align*}

  Finally, taking the sum of two tensors one can easily get
  $$
  {\rm{Ric}}_{ij} = \frac{1}{4} \bigl( g_{ij} +  \Phi_{ijk} \Phi^k  \bigr) = \frac{n-1}{4(n+1)}\Phi_{ij}.
  $$

  \section{Appendix}

  In the Appendix we establish estimates on the growth of the first and third derivatives.
  Unlike the previous sections, we don't assume that $n=2$.

   \subsection{First-order estimates}

 In this section we prove a priori  estimates for the  squared gradient norm
  $$
 \Phi_i \Phi^i = |\nabla \Phi|^2.
 $$

 \begin{lemma} Assume that $K \subset B_R(0)$. Then
 \label{rough-est-1-ord}
$$
|\nabla \Phi|^2  \le 2^{n-1} R^{2n}  e^{\Phi}.
$$
 \end{lemma}
 \begin{proof}
 $$
|\nabla \Phi|^2 = \Phi^{ij} \Phi_i \Phi_j
\le \frac{R^2}{\lambda_{\min}},
$$
where $\lambda_{\min}$ is the minimal Euclidean eigenvalue of $D^2 \Phi$.
By the result from \cite{K_part_I} the
maximal eigenvalue $\lambda_{max}$ is estimated by $2 R^2$.
Hence
$$
 e^{-\Phi} = \det D^2 \Phi \le \lambda_{\min} \lambda^{n-1}_{\max} \le  \lambda_{\min} (2 R^2)^{n-1}
$$
and we get the claim.
\end{proof}

 Now we prove yet another estimate which does not depend on the shape of $K$. The proof is based on an application of the maximum principle to the function
 $$
 \log F - \alpha \Phi,
 $$
 where $F=\Phi_i \Phi^i = |\nabla \Phi|^2$
 and $\alpha >1$.

 \begin{proposition}
 \label{nabla-phi-est}
 For every $\alpha>1$
 $$
 |\nabla \Phi|^2 = \Phi_i \Phi^i \le  \frac{\alpha(n+4) n}{(1-\alpha)^2} e^{\alpha(\Phi(x) - m)},
 $$
 where $m = \min_{x \in \mathbb{R}^n} \Phi(x)$.
 \end{proposition}
 \begin{proof}
 Since $\lim_{x \to \infty} \Phi(x) = + \infty$, it follows from lemma \ref{rough-est-1-ord}
 that the function
 $ \log F - \alpha \Phi$ attains its maxumum at some point  $x_0$. The following relations hold
 at this point:
 \begin{equation}
 \label{1order}
 \frac{F_i}{F} = \alpha \Phi_i,
 \end{equation}
 \begin{equation}
 \label{2order}
 \frac{L F}{F} -  \frac{|\nabla F|^2}{F^2} \le \alpha n.
 \end{equation}
 Taking into account the first equation we get $\frac{|\nabla F|^2}{F^2} = \alpha^2 |\nabla \Phi|^2 = \alpha^2 F$. Applying
 the formula $LF = 2n + 3 F + g_{ij} \Phi^i \Phi^j$, proven as in (L1) above, we get
 \begin{equation}
 \label{maxprinc}
 \frac{2n + 3 F + g_{ij} \Phi^i \Phi^j}{F} \le \alpha^2 F + \alpha n.
 \end{equation}
 Let us estimate $g_{ij} \Phi^i \Phi^j$. For every number $\lambda$ one has the following identity:
 \begin{align*}
 g_{ij} \Phi^i \Phi^j & = \Phi_{iab} \Phi_{j}^{ab} \Phi^i \Phi^j
 = \Bigl( \Phi_{iab} - \lambda \frac{\Phi_{ia} \Phi_b + \Phi_{ib} \Phi_a}{2}\Bigr)
 \Bigl( \Phi_{j}^{ab} - \lambda \frac{\Phi_{j}^{a} \Phi^b + \Phi_{j}^{b} \Phi^a}{2}\Bigr)  \Phi^i \Phi^j
 \\&
 + \lambda (\Phi_{ia} \Phi_b + \Phi_{ib} \Phi_a) \Phi_{j}^{ab} \Phi^i \Phi^j
 - \lambda^2  \Bigl(\frac{\Phi_{ia} \Phi_b + \Phi_{ib} \Phi_a}{2} \Bigr) \Bigl( \frac{\Phi_{j}^{a} \Phi^b + \Phi_{j}^{b} \Phi^a}{2}\Bigr) \Phi^i \Phi^j
 \end{align*}
 Note that the last term in the right hand side  equals $-\lambda^2 F^2$. To estimate the first term we use the inequality
 $Tr A^2 \ge \frac{1}{n} (Tr A)^2$ in the form $T_{ia}^b T_{j b}^a \geq T_{ia}^a T_{j b}^n / n$ in the sense of
 symmetric $2$-tensors, as well as the relation
  $$ \Phi^{ab} \bigl(\Phi_{iab} - \lambda \frac{\Phi_{ia} \Phi_b + \Phi_{ib} \Phi_a}{2}\bigr) = - (1+ \lambda) \Phi_i.$$
  To estimate the second term let us use
  (\ref{1order}): $\alpha F  \Phi_i = 2( \Phi_i - \frac{1}{2} \Phi_{iab} \Phi^a \Phi^b)$
and hence $\Phi_{i ab} \Phi^a \Phi^b = (2 - \alpha F) \Phi_i$.  Thus
  $$
   \lambda (\Phi_{ia} \Phi_b + \Phi_{ib} \Phi_a) \Phi_{j}^{ab} \Phi^i \Phi^j =2 \lambda (2- \alpha F)F.
  $$
  Finally, one gets
  $$
   g_{ij} \Phi^i \Phi^j  \ge 2 \lambda (2- \alpha F)F - \lambda^2 F^2 + \frac{1}{n} (1+ \lambda)^2 F^2
  $$
 Together with (\ref{maxprinc}) this implies
 $$
  \alpha n \ge   \frac{2n}{F} + 3 +  4 \lambda
  + F \Bigl(\frac{1}{n} (1+ \lambda)^2 - (\lambda + \alpha)^2  \Bigr).
 $$
 Set $\lambda = - \alpha$. We get
 $$
 F \le  \frac{\alpha(n+4) n}{(1-\alpha)^2}.
 $$
 This implies
 $$
 \log F - \alpha \Phi \le \log \Bigl[  \frac{\alpha(n+4) n}{(1-\alpha)^2}\Bigr]  - \alpha \min_{x} \Phi(x)
 $$
  and the result follows by the maximum principle.
\end{proof}

Since
$$
\Phi_i \Phi^i =  \sup_{\in \mathbb{R}^n: |v| =1}\Phi^2_v
= \sup_{e \in \mathbb{R}^n}\frac{\Phi^2_e}{\Phi_{ee}},
$$
where $|v|$ stand for the Riemannian length of the tangent vector $v$, we get immediately the following corollary.

\begin{corollary}
For every $e \in \mathbb{R}^n$ and $\alpha>1$ one has
$$
\Phi^2_e \le \Phi_{ee}  \frac{\alpha(n+4) n}{(1-\alpha)^2} e^{\alpha(\Phi(x) - m)}.
$$
\end{corollary}

\subsection{Third-order estimates}

\begin{lemma}
\label{LfAB}
Asume that $f$ is a non-negative smooth function on $\RR^n$ satisfying the inequality
$$
Lf \ge A f + B f^2
$$
for some $A,B>0$.
Then for every $\alpha>1$ the following inequality holds:
$$
f \le \frac{1}{B} \bigl[  (\alpha n - A)_{+} +  \frac{\alpha^3(n+4) n}{(1-\alpha)^2}) \bigr] e^{\alpha(\Phi(x) - m)}.
$$
\end{lemma}
\begin{proof}
Let us show first that there exists a number $c>0$ such that
\begin{equation}
\label{rough-fcphi}
f \le c e^{\alpha \Phi}.
\end{equation}
Fix a point $x_0$ where $\Phi(x_0)>1$  and set $R = (1+\varepsilon) \Phi(x_0)$.
We are looking for a maximum of the function
$$
f (R - \Phi)^2
$$
on the set $\{ \Phi \le R\}$. Applying the standard maximum principle, we get
$$
\frac{\nabla f}{f}
-2 \frac{\nabla \Phi}{R-\Phi} =0,
$$
$$
\frac{L f}{f}  -
\Bigl|\frac{\nabla  f}{f} \Bigr|^2 - 2 \frac{|\nabla \Phi|^2}{(R -\Phi)^2} \le 2  \frac{L\Phi}{R -\Phi}.
$$
Applying  the relation $L \Phi=n$, we get
\begin{equation}
\label{maxpr-cphi}
A + B f \le 6 \frac{|\nabla \Phi|^2}{(R -\Phi)^2} + \frac{2n}{R-\Phi}.
\end{equation}
Hence by  Lemma \ref{rough-est-1-ord}
$$
B f (R- \Phi)^2
\le c(n,K) e^{R} + 2n(R - \Phi) \le \tilde{C}(n, K)  e^R
$$
Using that $R = (1+\varepsilon) \Phi(x_0)$
we obtain (\ref{rough-fcphi}).

To get the claim let us apply the maximum principle to the function
$$
\log f - \alpha\Phi.
$$
Indeed, one has at the maximum point
$$
\frac{f_i}{f} = \alpha \Phi_i
$$
$$
\frac{L f}{f} - \frac{|\nabla f|^2}{f^2} \le  \alpha n.
$$
Then it follows from the assumption of the lemma and Proposition
 \ref{nabla-phi-est} that
$$
A+ B f \le   \alpha n +   \alpha^2  |\nabla \Phi|^2
\le  \alpha n +  \frac{\alpha^3(n+4) n}{(1-\alpha)^2} e^{\alpha(\Phi(x) - m)}.
$$
and one easily gets the desired inequality.
\end{proof}

  \begin{lemma}
  \label{3formvariation}
Let $T \colon (\mathbb{R}^n)^3 \to \mathbb{R}$ be a trilinear  form
which is symmetric
with repect to any permutation of the arguments.
There exists a unit vector $v$ such that
$$
T(v,v,v) =\sup_{e_1, e_2, e_3} T(e_1, e_2, e_3) \ge 0,
$$
where the supremum is taken over all unit vectors $e_1, e_2, e_3$.

In addition, for every unit vector $a \bot v$
$$
T(a,v,v)=0,
$$
$$
 T(v,v,v) \ge 2 T(v,a,a).
$$
  \end{lemma}
  \begin{proof} Let $e_1^0, e_2^0, e_3^0 \in S^{n-1}$ be such that
  $$
  \sup_{e_1, e_2, e_3} T(e_1, e_2, e_3)  = T(e^0_1, e^0_2, e^0_3).
  $$
We claim that $e^0_2 = \pm e_0^3$.
Indeed, by linear algebra, the supremum of the symmetric quadratic form
  $$
  (v,w) \to T(e^0_1, v, w)
  $$
  over unit vectors $v, w$ is the largest modulus of an eigenvalue and necessarily $v=\pm w$ are corresponding eigenvectors.
  Hence $e_0^2 = \pm e_0^3$. By the same argument $e^0_1 = \pm e^0_2 = \pm e^0_3$.
  It is clear that replacing some of $e^0_1, e^0_2, e^0_3$ by $-e^0_1,- e^0_2, -e^0_3$
  if necessary, we may chose $v$ satifying $T(e^0_1, e^0_2, e^0_3) = T(v,v,v)$.

  Next we note that for every $\varepsilon \in \mathbb{R}$ and $a \bot v$,
 \begin{align*}
   T(v,v,v) \ge   & T(\cos \varepsilon \, v + \sin \varepsilon \, a,\cos \varepsilon \, v + \sin \varepsilon \, a,\cos \varepsilon \, v + \sin \varepsilon \, a)
   \\&
   =T(v,v,v) +  3 \varepsilon T(a,v,v) + 3\varepsilon^2\bigl( T(v,a,a) - \frac{1}{2} T(v,v,v) \bigr) + o(\varepsilon^2).
  \end{align*}
  The claim follows by the maximality property of the unit vector $v$.
  \end{proof}

In what follows we apply the following formula for the Laplacian
of the 3-linear form $\Phi_{iab}$ (see  Lemma \ref{Lphi3}):
   $$
L \Phi_{iab}  =\frac{1}{2} \Phi_{iab}
- \frac{1}{2} \Phi^{l}_{ik} \Phi^{m}_{al} \Phi^{k}_{bm}
+ \frac{1}{4} \Bigl( g^{k}_{i} \Phi_{kab} +  g^{k}_{a} \Phi_{kib} +  g^{k}_{b} \Phi_{kia} \Bigr).
$$

\begin{proposition}
The function
\begin{equation}
f(x) =  \max_{e: |e|=1} \Phi_{eee}(x) 
\label{eq_354} \end{equation}
satisfies the following inequality
$$
    Lf  \ge
    \frac{1}{2} f + \frac{1}{4}  f^3.
    $$
\end{proposition}
\begin{proof}
Fix a point $x_0$.
According to the previous Lemma
$$
f(x_0) = \max_{e: |e|=1} \Phi_{eee}(x_0) =  \Phi_{vvv}(x_0)
$$
for some unit  vector $v \in TM(x_0)$.
Extend $v$ locally to a unit vector field $v(x)$ with the properties
 \begin{equation}
 \label{nabladelta}
  \nabla v=0, \ \Delta v=0
  \end{equation}
  at $x_0$  (see, for instance,  Theorem 4.6 of \cite{CK}). Clearly, $f (x) \ge   \Phi_{vvv}(x)$ and $f (x_0) =  \Phi_{vvv}(x_0)$.
  By the maximum principle
  $$
  L f(x_0)  \ge L \Phi_{vvv}(x)|_{x=x_0}.
  $$
  Applying (\ref{Lphi3}) and (\ref{nabladelta})
   one gets at $x_0$ (thus in any other point)
   $$
   L f \ge \Bigl[ \frac{1}{2} \Phi_{iab}
- \frac{1}{2} \Phi^{l}_{ik} \Phi^{m}_{al} \Phi^{k}_{bm}
+ \frac{1}{4} \Bigl( g^{k}_{i} \Phi_{kab} +  g^{k}_{a} \Phi_{kib} +  g^{k}_{b} \Phi_{kia} \Bigr) \Bigr] v^i v^a v^b.
   $$
   By the previous Lemma
   $\Phi_{avv}=0$  for every unit $a \bot v$.
   Hence
   $$
   Lf \ge \frac{1}{2} f + \frac{3}{4} g_{vv} f - \frac{1}{2} \sum_{i=1}^{n} \Phi^3_{v v_i v_i},
   $$
   where $v_i, 1 \le i \le n$ is an orthogomal frame  chosen in such a way that
  the quadratic form $(u,w) \to \Phi_{v u w}$ is diagonal in this frame and $v_1=v$.
  Note that
  $$
  g_{vv} = \sum_{i=1}^n \Phi^2_{v v_i v_i} = f^2 +  \sum_{i>1}^n \Phi^2_{v v_i v_i}.
  $$
  By the previous Lemma $f = \Phi_{vvv} \geq 2 \Phi_{v v_i v_i}$ for any $i > 1$, hence,
  $$
  f g_{vv} \ge f^3 + 2 \sum_{i>1}^n \Phi^3_{v v_i v_i}.
  $$
  Substituting this into the inequality for $Lf$
  $$
    Lf \ge \frac{1}{2} f + \frac{1}{2} g_{vv} f  + \frac{1}{4} (f^3 + 2 \sum_{i>1}^n \Phi^3_{v v_i v_i}) - \frac{1}{2} \sum_{i=1}^{n} \Phi^3_{v v_i v_i} \ge
    \frac{1}{2} f + \frac{1}{2} g_{vv} f  - \frac{1}{4} f^3.
  $$
  Since $f g_{vv} / 2 \geq f^3 / 2$ we obtain the claim.
\end{proof}

   In particular, it follows from this proposition that $f$ from (\ref{eq_354}) satisfies $ L( f^2 ) \ge \frac{1}{2}  f^4$.
   We immediately get from Lemma \ref{LfAB}
   \begin{corollary}
   \label{three-est}
   For every $\alpha>1$,
   $$
   \max_{e: |e|=1} \Phi^2_{eee} \le  c\frac{\alpha^3 n^2}{(1-\alpha)^2} e^{\alpha(\Phi(x) - m)}.
   $$
   where $c > 0$ is a universal constant.
 Equivalently, for any triplet of vectors $u,v,w$ (not necessary unit ones)
 $$
 \Phi^2_{uvw} \le  c \Phi_{uu} \Phi_{vv} \Phi_{ww} \frac{\alpha^3 n^2}{(1-\alpha)^2} e^{\alpha(\Phi(x) - m)}.
 $$
   \end{corollary}

   Similarly, Lemma \ref{lgij} implies the estimate $ L \|g\|  \ge  \|g\|  +\frac{1}{2} \|g\|^2$, where $
\| g \|= \sup_{e: |e|=1} g_{ee}
$ is the operator norm of $g$ in the metric $h$. From this we get

     \begin{corollary}
   For every $\alpha>1$,
   $$
   \max_{e: |e|=1} g^2_{ee} \le  c\frac{\alpha^3 n^2}{(1-\alpha)^2} e^{\alpha(\Phi(x) - m)},
   $$
   where $c > 0$ is a universal constant.
   \end{corollary}


\begin{thebibliography}{10}

 \bibitem{AGM}
 Artstein-Avidan~S.,  Giannopoulos~A., Milman~V.~D., Asymptotic geometric analysis. Part I, Mathematical Surveys and
Monographs, 202. AMS, Providence, RI, 2015.

  \bibitem{Aubin}
Aubin~T., Some nonlinear problems in Riemannian geometry,
Springer Monographs in Mathematics, 1998.

\bibitem{BBI} Burago~ D., Burago~Y., Ivanov~S.,
A course in metric geometry.
Graduate Studies in Mathematics, 33. AMS, Providence, RI, 2001.

\bibitem{BeBe} Berman~R. J., Berndtsson~ B.,  Real Monge-Amp\`ere equations and K\"ahler-Ricci solitons on toric log Fano varieties.
Ann. Fac. Sci. Toulouse Math. (6) 22 (2013), no. 4, 649--711.


\bibitem{CEKl}
Cordero-Erausquin~D.,  Klartag~B.,  Moment measures. J. Functional Analysis,  268 (12), (2015), 3834--3866.



\bibitem{CK}
 Chow~B., Knopf~D., The Ricci Flow: An Introduction.
Mathematical surveys and monographs, AMS, 2004.

\bibitem{Donaldson} Donaldson~S. K., K\"ahler geometry on toric manifolds, and some other manifolds  with
large symmetry. Handbook of geometric analysis. Adv. Lect. Math. (ALM), Vol. 7, No. 1,
Int. Press, Somerville, MA, (2008), 29--75.

\bibitem{K_moment} Klartag~B., { Poincar\'e inequalities and moment maps. }
Ann. Fac. Sci. Toulouse Math., 22(1), (2013), 1--41.


\bibitem{K_part_I} Klartag~B., {Logarithmically-concave moment measures I.}
Geometric Aspects of Functional Analysis, Vol. 2116 of the series Lecture Notes in Mathematics, (2014),231--260.



\bibitem{KK}  Klartag~B., Kolesnikov~A.V., Eigenvalue distribution of optimal transportation. Analysis \& PDE, 8(1) (2015), 33--55.

\bibitem{KK2} Klartag B., Kolesnikov A.V.,  Remarks on curvature in the transportation metric, Analysis Mathematica
43(1) (2017), 67--88.



\bibitem{Kol} Kolesnikov~A.V., Hessian metrics, CD(K,N)-spaces, and optimal transportation of log-concave measures. Discrete and Continuous Dynamical Systems - Series A. 34(4) (2014), 1511--1532.

\bibitem{KolMil} Kolesnikov A.V.,  Milman E., Riemannian metrics on convex sets with applications to Poincar{\'e} and log-Sobolev inequalities, Calc. Var. \& PDE's
(2016) 55--77.

 \bibitem{Legendre} Legendre~E., Toric K\"ahler--Einstein Metrics and Convex Compact Polytopes, J. Geom. Anal., 26(1) (2016), 399--427.

\bibitem{WangZhu} Wang~X.-J., Zhu~X.,
K\"ahler--Ricci solitons on toric manifolds with positive first Chern
class. Advances in Math., 188, (2004), 87--103.

  \end{thebibliography}
  \end{document}